\documentclass[12pt,a4paper]{article}
\usepackage{amssymb}

\usepackage{graphicx}
\usepackage{amsmath}

\usepackage{amsfonts}
\usepackage{amsthm,amscd}
\usepackage{graphicx}
\usepackage{amsmath}
\usepackage{amssymb}
\usepackage{amstext}

\newtheorem{theorem}{Theorem}

\newtheorem{corollary}[theorem]{Corollary}

\newtheorem{definition}[theorem]{Definition}

\newtheorem{lemma}[theorem]{Lemma}

\newtheorem{proposition}[theorem]{Proposition}

\begin{document}

\title{Zeta measures  and Thermodynamic Formalism for temperature zero}

\author{Artur O. Lopes (*)\,and Jairo K. Mengue (**)}

\date{\,\, \,Instituto de Matem\'atica, UFRGS - Porto Alegre, Brasil.}

\maketitle

\begin{abstract}

We address  the analysis of the following problem: given a real H\"older potential $f$ defined on the Bernoulli  space and $\mu_f$ its equilibrium state, it is known  that this shift-invariant probability can be weakly approximated by probabilities in periodic orbits associated to certain zeta functions.

Given a H\"older function $f>0$ and a value $s$ such that $0<s<1$,
we can associate a shift-invariant probability
$\nu_{s}$ such that for each continuous function $k$ we have
\[\int\, k \,d\nu_{s}=\frac{\sum_{n=1}^{\infty}\sum_{x\in Fix_{n}}e^{sf^{n}(x)-nP(f)}\frac{k^{n}(x)}{n}}{\sum_{n=1}^{\infty}\sum_{x\in Fix_{n}}e^{sf^{n}(x)-nP(f)}},\]
where $P(f)$ is the pressure of $f$, $Fix_n$ is the set of solutions of $\sigma^n(x)=x$, for any $n\in \mathbb{N}$, and $f^{n}(x) = f(x) + f(\sigma(x)) +... + f(\sigma^{n-1} (x)).$

We call  $\nu_{s}$ a zeta probability for $f$ and $s$, because
it can be obtained in a natural way from the dynamical zeta-functions.
From the work of W. Parry and M. Pollicott it is known  that $\nu_s \to \mu_{ f}$, when $s \to 1$.
We consider for each value $c$ the  potential $c\, f$ and the corresponding equilibrium state $\mu_{c f}$. What happens with $\nu_{s}$ when $c$ goes to infinity and $s$ goes to one? This question is related to the problem of  how to approximate the maximizing probability for $f$ by probabilities on periodic orbits. We study this question and also present here  the deviation function $I$ and Large Deviation Principle for this limit $c\to \infty,\, s\to 1$.
We will make an assumption: for some fixed $L$ we have  $\lim_{c\to \infty,\,s\to 1}\, c(1-s)= L>0$. We do not assume here the maximizing probability for $f$  is unique in order to get the L.  D. P.
The deviation function described here is different from the one obtained by A. Baraviera, A. O. Lopes and Ph. Thieullen.
\vspace{0.2cm}
 
to appear in Bull. of the Bras. Math. Soc.
\vspace{0.2cm}

\end{abstract}

\maketitle

\section{Introduction}

We denote by  $X=\{1,...,d\}^\mathbb{N}$ the Bernoulli space with $d$ symbols. This is a compact metric space when one considers the usual metric
\[d(x,y) = d_{\theta}(x,y)=\theta^{N},\ \ x_{1}=y_{1},...,x_{N-1}=y_{N-1},x_{N}\neq y_{N},\]
with $\theta$ fixed $0<\theta<1$.

The results we will derive here are also true for shifts of finite type, but in order to simplify the notation, we will consider here in our proofs just case of the full Bernoulli space $X$.

We consider the Borel $\sigma-$algebra over $X$ and  denote by $\cal{ M}$ the set of invariant probabilities for the  shift. $F_{\theta}$ denotes the set of real Lipschitz functions over $X$. There is no big difference between H\"older and Lipschitz in this case (see page 16 in \cite{PP}).

Here we work with a strictly positive function  $f$ in $F_{\theta}$. We denote respectively
$$\beta(f):=\sup_{\nu \in \cal{M}}\int f d\nu,$$
$$M_{max}(f):=\{\nu \in \cal{M}: \int \text{\it{f}\,  d}\nu = \beta(\text{{\it f}\,\,})\},$$
$$h_{f}:= \sup\{h_{\nu} :\nu \in M_{max}({\it f})\}.$$

A probability $ \mu_\infty$ which $f$-integral attains the maximum value $M_{max}(f)$ will be called a $f$-maximizing probability. We refer the reader to \cite{Jenkinson1} \cite{CLT} \cite{HY} \cite{Le} \cite{BLT} \cite{Bousch1} for general properties of such probabilities.

As usual, given a real continuous function $k$ over $X$, and $x\in X$, the number $k^n(x)$ denotes
$k(x)+ k(\sigma(x))+ k(\sigma^2 (x)) +... +k(\sigma^{n-1} (x)) .$

We denote by $\text{Fix}_n$ the set of solutions $x$ to the equation $\sigma^n (x)=x$ and $P(f)$ is the pressure for $f$.

We address the reader to \cite{P1} \cite{PP} for
general properties of Thermodynamic Formalism, zeta functions and the procedure of approximating Gibbs states by probabilities with support on periodic orbits.

Following  \cite{P1} \cite{PP} (see also the last section) we consider probabilities $\mu_{c,s}$ in ${\cal M}$ such that for any continuous function $k:X \rightarrow \mathbb{R}$
\[\int\, k \,d\mu_{c,s}=\frac{\sum_{n=1}^{\infty}\sum_{x\in Fix_{n}}e^{c\,s\,f^{n}(x)-n\,P(c\,f)}\frac{k^{n}(x)}{n}}{\sum_{n=1}^{\infty}\sum_{x\in Fix_{n}}e^{c\,s\,f^{n}(x)-n\,P(c\,f)}},\]
where $c>0$, $s\in(0,1)$.

The above expression is well defined because $P(csf-P(cf))<0,$ and appear naturally when we work with the zeta function
\[\zeta(s,z)=\exp \left(\sum_{n=1}^{\infty}\frac{1}{n}\sum_{x\in Fix_{n}}e^{c\,s\,f^{n}(x)-n\,P(c\,f) + zk^{n}(x)}\right).\]

Following W. Parry and M. Pollicott \cite{PP} it is known that when $c$ is fixed and $s\to 1$, one gets that $\mu_{c,s}$ weakly converges to the Gibbs state for $cf$. We prove this here (see Lemma \ref{sto1}), but we will be really interested in analyzing $\mu_{c,s}$, when  $s\rightarrow 1$ and $c\rightarrow\infty$. Such limit is the maximizing probability $\mu_\infty$, when this is unique in $M_{max}(f)$ (see next theorem).
\bigskip

In order to simplify the notation $k$ will also represent the characteristic function of a general cylinder which will be also called $k$.

In the present setting, a Large Deviation Principle should be the identification  of a function $I:X \to \mathbb{R}$, which is non-negative, lower semi-continuous and such that, for any cylinder $k\subset X$, we have
\[\lim_{c\to\infty, \,s\to 1} \frac{1}{c}\log(\mu_{c,s}(k))=-\inf_{x\in k}I(x).\]

We point out that we will need same  care in the way we consider the limits $s\to 1$ and $c \to \infty$. We will assume that in the above limit the values $c$ and $s$ are related by some constrains (which are in a certain sense natural).

A general reference for Large Deviation properties and theorems is \cite{DZ}.

We point out that
\[P(cf)=c\beta(f)+\epsilon_{c},\]
where $\epsilon_{c}$ decreases to $h_{f}$ when $c\to \infty$ (which was defined above) \cite{CG}.

In Thermodynamic Formalism and Statistical Mechanics $c=\frac{1}{T}$, where $T$ is temperature.  In this sense, to analyze the limit behavior of Gibbs states $\mu_{cf}$ when $c\to \infty$, corresponds to analyze a system under temperature zero for the potential $f$ (see also \cite{LMST}).

It is known that there exists certain Lipschitz potentials $f$ such that the sequence $\mu_{c\, f}$
does not converge to any probability when $c \to \infty$ \cite{CH}. We will not assume the maximizing probability is unique for the potential  $f$ in order to  get  the L. D. P.

\begin{definition}
We define the function $I(x)$ in the periodic points $x\in \text{PER}$ by:
\[I(x):=n_{x}\left(\beta(f) - \frac{f^{n_{x}}(x)}{n_{x}}\right),\]
where $n_{x}$ is the minimum period of  $x$, and $f>0$ is Lipschitz.
\end{definition}

We need some properties of  $I$. We show in section 2

\begin{lemma}\label{lema}

\[\inf_{x\in \text{PER}}I(x)=0.\]
\end{lemma}

The next result is not a surprise:

\begin{theorem}\label{teo}
Suppose $f>0$ is Lipschitz.
 When $c\rightarrow\infty$ and $s\rightarrow 1$, any accumulation point of $\mu_{c,s}$ is in $M_{max}(f)$.
Moreover, if $g:X\rightarrow \mathbb{R}$ is a continuous function $c_{j}\rightarrow\infty$ and $s_{j}\rightarrow 1$ are such that there exist
\[\lim_{j\rightarrow\infty}\mu_{c_{j},s_{j}}(g),\]
then this limit is $\int g d\mu$ for some accumulation point $\mu$ of  $\mu_{c,s}$ in the weak* topology.
\newline

In particular, if $\mu_\infty$ is unique, then for any continuous  $g:X\rightarrow \mathbb{R}$
\[\lim_{c\rightarrow\infty, s\rightarrow 1}\int g\, d\mu_{c,s}=\int g d\mu_\infty.\]
\end{theorem}

The main result of our paper is a Large Deviation Principle for $\mu_{c,s}$:

\begin{theorem}\label{ldt}
Suppose $f>0$ is Lipschitz. Then, for any fixed $L>0$ (it is allowed $L=\infty$), and for all cylinder $k\subset X$
\[\lim_{c(1-s)\rightarrow L} \frac{1}{c}\log(\mu_{c,s}(k))=-\inf_{x\in k, \,x \in \text{Per}}I(x).\]
The same is true if we have:
\[\liminf_{c\to\infty, \, s\to1}c(1-s)=L>0.\]
\end{theorem}

We going to extend $I$ (which was defined just for periodic orbits) to $\widetilde{I}$, defined on the all set $X$, which preserve the infimum of $I$ in  each cylinder, and, which is also lower semi-continuous and non-negative (see sections 4 and 5).

Finally, we can get the following result:
\begin{corollary}
Suppose $f>0$ is Lipschitz. Then, there is a Large Deviation Principle with deviation function $\widetilde{I}$:
for fixed $L>0$, and for any cylinder  $k\subset X$
\[\lim_{c(1-s)\rightarrow L}\frac{1}{c}\log \mu_{c,s}(k) = - \inf_{x\in k}\widetilde{I}(x),\]
 where $\widetilde{I}$ is lower semi-continuous and non-negative.
 \newline
 The same is true if we have:
\[\liminf_{c\to\infty, \, s\to1}c(1-s)=L>0.\]
\end{corollary}

In  \cite{BLT} it is assumed that the maximizing probability $\mu_\infty$ is unique.
The equilibrium probabilities $\mu_{c\, f}$ for the real Lipschitz potential $cf$ converge to $\mu_\infty$ and it is presented in \cite{BLT} a L.D.P. for such setting (a different deviation function). The deviation function $I_{BLT}$ in that paper is lower semi-continuous but can attains the value $\infty$ in some periodic points. Under the assumption  $\lim_{c\to \infty,\,s\to 1}\, c(1-s)= L>0$, we can show that the deviation rates in cylinders described here by $I$ are different from the ones in \cite{BLT} which are described by $I_{BLT}$. This is  described in  section 5 Proposition 15.
Finally, Proposition $\ref{zero}$ in section 6 shows that if $c(1-s)\rightarrow 0$ with a certain speed, then $\mu_{c,s}$ have a L.D.P., but the deviation function is $I_{BLT}$ (not  $I$).  We want to present a sufficient  analytic estimate that allows  one to find $s_c$ as a function of $c$ in such way this happens.

In the last section we also study the invariant probabilities $\pi_{c,N}$ and $\eta_{c,N}$ given respectively by

$$
\int k d\pi_{c,N}= \frac{\sum_{n=1}^{N}\sum_{x\in Fix_{n}}e^{c\,f^{n}(x)-n\,P(c\,f)}\frac{k^{n}(x)}{n}}{\sum_{n=1}^{N}\sum_{x\in Fix_{n}}e^{c\,f^{n}(x)-n\,P(c\,f)}},$$
and
$$\int k d\eta_{c,N}= \frac{\sum_{n=1}^{N}\sum_{x\in Fix_{n}}e^{c\,f^{n}(x)}\frac{k^{n}(x)}{n}}{\sum_{n=1}^{N}\sum_{x\in Fix_{n}}e^{c\,f^{n}(x)}},$$

where $c>0$, $N\in \mathbb{N}$.

We show that when $N \to \infty$ these probabilities converge weakly to $\mu_{cf}$ and when $c,N \to \infty$ they satisfies a result analogous to Theorem \ref{teo}, with small modifications on the proof. Also, if $N/c \to 0$, then $\pi_{c,N}$ have a L.D.P. which the same deviation function $\widetilde{I}$ above.

We point out that it follows from the methods we describe in this paper the following property:
given $f, \, f_{\lambda} \in F_{\theta}$, such that $f_{\lambda}\to f$ uniformly when $\lambda\to \infty$, suppose
there exist the weak$^*$ limit
$$\lim_{j\rightarrow\infty} \pi_{c_{j},N_{j},f_{\lambda_{j}}}= \nu,$$
$c_{j},N_{j},f_{\lambda_{j}}\to \infty$, then $\nu$ is a maximizing measure for $f$. Moreover, if we take first $N\to \infty$ and after $c_j, \lambda_j\to \infty$, then we get: any weak$^*$ accumulation point of of $\mu_{c_\lambda f_\lambda}$ is a maximizing probability for $f$.

In particular given $f$ one can consider for each $m$ a $f_m$ approximation which depend just on the first $m$ coordinates as in Proposition 1.3  \cite{PP}. We point out that for each $f_m$ the eigenvalues, eigenvectors, pressure,  etc... can be obtained via the classical Perron Theorem for positive  matrices \cite{PP} \cite{Sp}.

In the same way, if when $c_{j} ,\lambda_{j}\rightarrow \infty$, $s_j \to 1$, there exists the limit
 $$ \lim_{j\rightarrow\infty}\mu_{c_{j}f_{\lambda_{j}},s_j}=\nu,$$
 then $\nu$ is maximizing for $f$.

This work is part of the thesis dissertation of the second author in Prog. Pos-Grad. Mat. - UFRGS.

\section{Proof of Lemma \ref{lema}}

We want to show that
\[\inf_{x\in \text{PER}}I(x)=0.\]

We will need the following lemma (\cite{Atkinson} \cite{Mane}):
\begin{lemma}

Given a Borel measurable set $A$, a continuous $f:X \to \mathbb{R}$, and an ergodic probability $\nu$, with $\nu(A)>0$, there exists $p\in A$ such that for all $\epsilon>0$, there exists an integer $N>0$ which satisfies $\sigma^{N}(p)\in A$ and
\[\left|\sum_{i=0}^{N-1}f(\sigma^{i}(p)) - N\int f d\nu \right|<\epsilon.\]

The set of such $p\in A$ has full measure.
\end{lemma}

Now we will present the proof of Lemma $\ref{lema}$

\begin{proof} $M_{max}(f)$ is a compact convex set which contains at least one ergodic probability $\nu$.
\newline
Then,
\[\beta(f)=\int f d\nu \ \ \ \ \text{and} \ \ \ \ I(x) = |f^{n_{x}}(x)-n_{x}\int fd\nu|. \]

It is enough to show that for any  $\epsilon>0$, there exists $x\in \text{PER}$ such that
\[I(x) = |f^{n_{x}}(x)-n_{x}\int fd\nu|<\epsilon.\]

As $f\in F_{\theta}$, there exists a constant $C>0$ such that for any  $x,y \in X$:
\[|f(x)-f(y)|<Cd(x,y).\]

We fix $j$ such that
\[C\theta^{j}\frac{\theta}{1-\theta}<\epsilon/2.\]

 There exists a cylinder $k_{j}$ of size $j$ such that $\nu(k_{j})>0$. Using the last lemma with $A=k_{j}$ we are able to get a point $p\in k_{j}$ and an integer $N>0$ such that $\sigma^{N}(p) \in k_{j}$ and
\[ \left|\sum_{i=0}^{N-1}f(\sigma^{i}(p)) - N\int f d\nu \right|<\epsilon/2.\]

 It follows that $p$ is of the form $p=p_{1}...p_{N}p_{1}...p_{j}...$. Now if we consider the periodic point $x$ given by   repeating successively  the word
\[x=\overline{p_{1}...p_{N}},\]
then we get

\begin{align*}
\left|\sum_{i=0}^{N-1}f(\sigma^{i}(p)) - \sum_{i=0}^{N-1}f(\sigma^{i}(x))\right| &\leq \sum_{i=0}^{N-1} \left|f(\sigma^{i}(p)) - f(\sigma^{i}(x))\right| \\
&\leq C(d_{\theta}(p,x) +...+ d_{\theta}(\sigma^{N-1}(p),\sigma^{N-1}(x)))\\
&\leq C(\theta^{N+j} +...+ \theta^{j})\\
&< C\theta^{j}(1+\theta+...)\\
&= C\theta^{j}\frac{\theta}{1-\theta} < \epsilon/2.
\end{align*}

It follows that
\begin{align*}
I(x) &= \left|\sum_{i=0}^{n_{x}-1}f(\sigma^{i}(x)) - n_{x}\int f d\nu \right|\\
&\leq \left|\sum_{i=0}^{N-1}f(\sigma^{i}(x)) - N\int f d\nu \right|\\
&\leq \left|\sum_{i=0}^{N-1}f(\sigma^{i}(p)) - \sum_{i=0}^{N-1}f(\sigma^{i}(x))\right| + \left|\sum_{i=0}^{N-1}f(\sigma^{i}(p)) - N\int f d\nu \right|\\
&\leq \epsilon. \ \ \ \ \ \
\end{align*}
\end{proof}
\section{Proof of Theorem $\ref{teo}$}

We begin with an auxiliary lemma:

\begin{lemma}\label{aux}
Suppose $f>0$ is Lipschitz. Then,
\begin{equation}
\liminf_{c\rightarrow\infty, s\rightarrow 1}\mu_{c,s}(f) \geq \beta(f). \label{1}
\end{equation}
\end{lemma}

\begin{proof}
We write
\[P(cf)=c\beta(f) + \epsilon_{c},\]
where $\epsilon_{c}$ decrease to $h_{f}$ when $c\to \infty$ \cite{CG}.

Fix $\epsilon>0$.
We define
\[A_{n}=\{x\in Fix_{n}: \frac{f^{n}(x)}{n} < \beta(f) - \epsilon\},\]
\[B_{n}=\{x\in Fix_{n}: \frac{f^{n}(x)}{n} \geq \beta(f) - \epsilon\}.\]

It follows that for $c>>0$:
\begin{align*}
\sum_{n=1}^{\infty}\sum_{x\in A_{n}}e^{csf^{n}-nP(cf)} &\leq \sum_{n=1}^{\infty}\sum_{x\in A_{n}}e^{csn(\beta(f) - \epsilon)-nc\beta(f) -n\epsilon_{c}}\\
&=  \sum_{n=1}^{\infty}\sum_{x\in A_{n}}e^{nc(s-1)\beta(f) - ncs\epsilon -n\epsilon_{c}}\\
&\leq \sum_{n=1}^{\infty}\sum_{x\in A_{n}}e^{-ncs\epsilon}\\
&\leq  \sum_{n=1}^{\infty}e^{- ncs\epsilon + n\log(d)}\\
&= \frac{e^{- cs\epsilon + \log(d)}}{1-e^{- cs\epsilon + \log(d)}},
\end{align*}
and, with a similar reasoning,
\[\sum_{n=1}^{\infty}\sum_{x\in A_{n}}e^{csf^{n}-nP(cf)}\frac{f^{n}}{n} \leq \frac{e^{- cs\epsilon + \log(d)}}{1-e^{- cs\epsilon + \log(d)}}\left(\beta(f) - \epsilon\right).\]

By the other side, by lemma $\ref{lema}$, there exists a periodic point $x$ such that:
$$I(x)=n_{x}\left(\beta(f) - \frac{f^{n_{x}}(x)}{n_{x}}\right)<\epsilon/2.$$
Therefore,
\begin{align}
\sum_{n=1}^{\infty}\sum_{x\in B_{n}}e^{csf^{n}-nP(cf)} &\geq e^{csf^{n_{x}}(x)-n_{x}P(cf)} \nonumber \\
&=  e^{csf^{n_{x}}(x)-n_{x}c\beta(f) -n_{x}\epsilon_{c}} \label{2.1}\\
&= e^{-csn_{x}\left(\beta(f) - \frac{f^{n_{x}}(x)}{n_{x}}\right)+c(s-1)n_{x}\beta(f) -n_{x}\epsilon_{c}} \label{2.2}\\
&= e^{-csI(x)+c(s-1)n_{x}\beta(f) -n_{x}\epsilon_{c}} \label{2.3} \\
&\geq e^{-cs\epsilon/2 +c(s-1)n_{x}\beta(f) -n_{x}\epsilon_{c}}, \nonumber
\end{align}

and, with a similar reasoning,
\[\sum_{n=1}^{\infty}\sum_{x\in B_{n}}e^{csf^{n}-nP(cf)}\frac{f^{n}}{n} \geq e^{-cs\epsilon/2 +c(s-1)n_{x}\beta(f) -n_{x}\epsilon_{c}}\left(\beta(f) - \epsilon\right).\]

It follows that
\begin{align*}
\frac{\sum_{n=1}^{\infty}\sum_{x\in A_{n}}e^{csf^{n}-nP(cf)}\frac{f^{n}}{n}}{\sum_{n=1}^{\infty}\sum_{x\in B_{n}}e^{csf^{n}-nP(cf)}\frac{f^{n}}{n}}
& \leq \frac{e^{- cs\epsilon + \log(d)}}{1-e^{- cs\epsilon + \log(d)}}\frac{1}{e^{-cs\epsilon/2 +c(s-1)n_{x}\beta(f) -n_{x}\epsilon_{c}}}\\
&= \frac{e^{-cs\epsilon + \log(d) + cs\epsilon/2 -c(s-1)n_{x}\beta(f) +n_{x}\epsilon_{c}}}{1-e^{- cs\epsilon + \log(d)}}\\
&= \frac{e^{-c(s\epsilon/2 +(s-1)n_{x}\beta(f)) +\log(d) + n_{x}\epsilon_{c}}}{1-e^{- cs\epsilon + \log(d)}}\stackrel{s\to1, c\to\infty}{\to} 0.
\end{align*}

Finally, in the same way
\[\lim_{c\rightarrow\infty, s\rightarrow 1}
\frac{\sum_{n=1}^{\infty}\sum_{x\in A_{n}}e^{csf^{n}-nP(cf)}}{\sum_{n=1}^{\infty}\sum_{x\in B_{n}}e^{csf^{n}-nP(cf)}} = 0.\]

It follows that
\begin{align*}
\liminf_{c\rightarrow\infty, s\rightarrow 1} \frac{\sum_{n=1}^{\infty}\sum_{x\in Fix_{n}}e^{csf^{n}-nP(cf)}\frac{f^{n}}{n}}{\sum_{n=1}^{\infty}\sum_{x\in Fix_{n}}e^{csf^{n}-nP(cf)}}
&= \liminf_{c\rightarrow\infty, s\rightarrow 1} \frac{\sum_{n=1}^{\infty}\sum_{x\in B_{n}}e^{csf^{n}-nP(cf)}\frac{f^{n}}{n}}{\sum_{n=1}^{\infty}\sum_{x\in B_{n}}e^{csf^{n}-nP(cf)}}\\
&\geq \beta(f) - \epsilon.
\end{align*}

As we consider a general $\epsilon > 0$, then we get

\[\liminf_{c\rightarrow\infty, s\rightarrow 1} \frac{\sum_{n=1}^{\infty}\sum_{x\in Fix_{n}}e^{csf^{n}-nP(cf)}\frac{f^{n}}{n}}{\sum_{n=1}^{\infty}\sum_{x\in Fix_{n}}e^{csf^{n}-nP(cf)}} \geq \beta(f).\]
\end{proof}

Now we can show the proof of Theorem $\ref{teo}$.

\begin{proof}
 Suppose $\mu$ is an accumulation point of  $\mu_{c,s}$. Then, $\mu$ is a $\sigma-$invariant probability and by last lemma
\[\mu(f)\geq \beta(f),\]
and from this follows that $\mu \in M_{max}(f)$.
\newline

Now we fix a continuous function $g$ and sequences $c_{j}\rightarrow\infty$ and
$s_{j}\rightarrow 1$, such that there exists
\[\lim_{j\rightarrow\infty}\mu_{c_{j},s_{j}}(g).\]

By the diagonal Cantor argument, there exists a subsequence $\{j_{i}\}$, such that there exists
\[\mu(k) := \lim_{i\rightarrow\infty}\mu_{c_{j_{i}},s_{j_{i}}}(k),\]
for any cylinder $k$.

We will show that for any $h$, there exists the limit
\[\lim_{i\rightarrow\infty}\mu_{c_{j_{i}},s_{j_{i}}}(h).\]

 Given $\epsilon>0$, as $X$ is compact, there exists functions $k_{1}$ and $k_{2}$, that can be written as linear combinations of  characteristic functions of   cylinders, such that for all $x\in X$
\[k_{1}(x) \leq h(x) \leq k_{2}(x) \leq k_{1}(x) + \epsilon.\]
It follows that
\begin{align*}
\limsup_{i\rightarrow\infty}\mu_{c_{j_{i}},s_{j_{i}}}(h) &\leq \lim_{i\rightarrow\infty}\mu_{c_{j_{i}},s_{j_{i}}}(k_{2}) \leq \lim_{i\rightarrow\infty}\mu_{c_{j_{i}},s_{j_{i}}}(k_{1}) + \epsilon\\
&\leq \liminf_{i\rightarrow\infty}\mu_{c_{j_{i}},s_{j_{i}}}(h) + \epsilon.
\end{align*}
Therefore
\[\liminf_{i\rightarrow\infty}\mu_{c_{j_{i}},s_{j_{i}}}(h) = \limsup_{i\rightarrow\infty}\mu_{c_{j_{i}},s_{j_{i}}}(h).\]

It follows that for any continuous function $h$ there exists the limit
\[\mu(h):=\lim_{i\rightarrow\infty}\mu_{c_{j_{i}},s_{j_{i}}}(h).\]

Therefore, $\mu$ is an accumulation point of the  $\mu_{c,s}$. Moreover,
\[\lim_{j\rightarrow\infty}\mu_{c_{j},s_{j}}(g)=\lim_{i\rightarrow\infty}\mu_{c_{j_{i}},s_{j_{i}}}(g)=\mu(g).\]
\end{proof}


\section{Proof of theorem $\ref{ldt}$}

We will show that: for any fixed $L>0$ (it can be that $L=\infty$), and any cylinder $k$
\[\lim_{c(1-s)\rightarrow L} \frac{1}{c}\log(\mu_{c,s}(k))=-\inf_{x\in k, \,x \in \text{PER}}I(x).\]
\textbf{Remark:}
 As we point out in the introduction we have to consider $c\rightarrow\infty$ and $s\rightarrow 1$. The hypothesis $c(1-s)\rightarrow L$ can be understood as a constraint on the speed such that simultaneously  $c\rightarrow\infty$ and $s\rightarrow 1$: that is, $c(1-s)\rightarrow L$.

The proof presented here also covers the case where we assume
\[\liminf_{c\rightarrow\infty, s\rightarrow 1} c(1-s) = L > 0,\]
and, it is not really necessary that $c(1-s)\rightarrow L$.

Now we will present the proof of theorem $\ref{ldt}$
\begin{proof}

Remember that we denote a cylinder $k$ and the indicator function of this set also by $k$.

It is enough to show that for any fixed cylinder $k$
\[\lim_{c(1-s)\rightarrow L} \frac{1}{c}\log \sum_{n=1}^{\infty}\sum_{y\in Fix_{n}}e^{csf^{n}(y)-nP(cf)}\frac{k^{n}(y)}{n} = -\inf_{x\in k, \, x \in \text{PER}}I(x),\]
because, by taking $k=X$, we will get
\[\lim_{c(1-s)\rightarrow L} \frac{1}{c}\log \sum_{n=1}^{\infty}\sum_{y\in Fix_{n}}e^{csf^{n}-nP(cf)} = -\inf_{x\in \text{PER}}I(x) = 0.\]

First we will show the lower (large deviation) inequality
\[\liminf_{c(1-s)\rightarrow L} \frac{1}{c}\log \sum_{n=1}^{\infty}\sum_{y\in Fix_{n}}e^{csf^{n}-nP(cf)}\frac{k^{n}}{n} \geq -\inf_{x\in k\,,\, x \in \text{PER}}I(x),\]
(for this part it is just enough to assume $c\rightarrow\infty$ and $s\rightarrow 1$).
\newline
Consider a generic  point  $x\in k$ which is part of a periodic orbit $\{x,...,\sigma^{(n_{x}-1)}x\}$. Therefore,
\begin{align*}
\sum_{n=1}^{\infty}\sum_{y\in Fix_{n}}e^{csf^{n}(y)-nP(cf)}\frac{k^{n}(y)}{n} &\geq
 \sum_{\{x,...,\sigma^{(n_{x}-1)}x\}}e^{csf^{n_{x}}-n_{x}P(cf)}\frac{k^{n_{x}}}{n_{x}}\\
&\geq  e^{csf^{n_{x}}(x)-n_{x}P(cf)}k^{n_{x}}(x)\\
&\geq  e^{csf^{n_{x}}(x)-n_{x}P(cf)}\\
&= e^{-csI(x) + n_{x}c(s-1)\beta(f) - n_{x}\epsilon_{c}} \ \ \ \ (\text{by} \ \ (\ref{2.1}), (\ref{2.2}), (\ref{2.3})).
\end{align*}

From this follows that
\begin{align*}
\liminf_{c(1-s)\rightarrow L} & \frac{1}{c}\log \sum_{n=1}^{\infty}\sum_{y\in Fix_{n}}e^{csf^{n}-nP(cf)}\frac{k^{n}}{n}\\
&\geq \liminf_{c(1-s)\rightarrow L} \frac{1}{c}\log \left(e^{-csI(x) + n_{x}c(s-1)\beta(f) - n_{x}\epsilon_{c}}\right)\\
&\geq \liminf_{c(1-s)\rightarrow L}-sI(x) + n_{x}(s-1)\beta(f) - \frac{n_{x}\epsilon_{c}}{c}\\
&= -I(x).
\end{align*}

As we take  $x$ as a generic periodic point in  $k$ we finally get

\[\liminf_{c(1-s)\rightarrow L} \frac{1}{c}\log \sum_{n=1}^{\infty}\sum_{y\in Fix_{n}}e^{csf^{n}-nP(cf)}\frac{k^{n}}{n}\geq-\inf_{x\in k, \, x \in \text{PER}}I(x).\]

Now we will show the upper (large deviation) inequality
\[\limsup_{c(1-s)\rightarrow L} \frac{1}{c}\log \sum_{n=1}^{\infty}\sum_{y\in Fix_{n}}e^{csf^{n}-nP(cf)}\frac{k^{n}}{n}\leq-\inf_{x\in k, \, x \in \text{PER}}I(x).\]
We will denote the value $\inf_{x\in k, \, x \in \text{PER}}I(x)$ by $I$.

 Consider a fixed $\delta>0$. As $f>0$ and $f$ is continuous, there exists a constant $|f|_{-}>0$ such that $f>|f|_{-}$.
As $c(1-s)\rightarrow L>0$ (or just considering $\liminf c(1-s) = L$) there exists $\psi>0$
such that for $c$ big enough  $c(1-s)> 2\psi$. As $\epsilon_{c}=P(cf)-c\beta(f)$ decrease to
$h_{f}$, we can also suppose that $c$ is such that  $\epsilon_{c\delta}<h_{f}+\psi|f|_{-}$.
Therefore, there exists $c_{0}$ such that for $c\geq c_{0}$
$$ c(1-s)|f|_{-}+h_{f}> h_{f} + 2\psi|f|_{-} > \epsilon_{c_{0}\delta} + \psi|f|_{-}.$$
The conclusion is that for such  $c\geq c_{0}$:
\begin{align*}
\sum_{n=1}^{\infty}\sum_{y\in Fix_{n}}e^{csf^{n}-nP(cf)}\frac{k^{n}(y)}{n} &=
\sum_{n=1}^{\infty}\sum_{y\in Fix_{n}}e^{cf^{n}(y) + c(s-1)f^{n}(y) - cn\left(\beta(f)\right) -n\epsilon_{c}}\frac{k^{n}(y)}{n}\\
&\leq \sum_{n=1}^{\infty}\sum_{y\in Fix_{n}}e^{-cn\left(\beta(f) -\frac{f^{n}(y)}{n}\right)-c(1-s)n|f|_{-}-nh_{f}}\frac{k^{n}(y)}{n}\\
&\leq \sum_{n=1}^{\infty}\sum_{y\in Fix_{n}}e^{-cn\left(\beta(f)-\frac{f^{n}(y)}{n}\right)((1-\delta)+\delta)-c(1-s)n|f|_{-}-nh_{f}}\frac{k^{n}(y)}{n}\\
&\leq \sum_{n=1}^{\infty}\sum_{y\in Fix_{n}}e^{-cI(1-\delta)-cn\left(\beta(f)-\frac{f^{n}(y)}{n}\right)\delta - c(1-s)n|f|_{-}-nh_{f}}\frac{k^{n}(y)}{n}\\
&\leq e^{-cI(1-\delta)}\sum_{n=1}^{\infty}\sum_{y\in Fix_{n}}e^{-nc\delta\left(\beta(f) - \frac{f^{n}(y)}{n}\right) -n\epsilon_{c_{0}\delta} - n\psi|f|_{-} }\\
&\leq e^{-cI(1-\delta)}\sum_{n=1}^{\infty}\sum_{y\in Fix_{n}}e^{-nc_{0}\delta\left(\beta(f)-\frac{f^{n}(y)}{n}\right) -n\epsilon_{c_{0}\delta} - n\psi|f|_{-}  }\\
&\leq e^{-cI(1-\delta)}\sum_{n=1}^{\infty}\sum_{y\in Fix_{n}}e^{c_{0}\delta f^{n}(y) -nP(c_{0}\delta f)-n\psi|f|_{-}}.
\end{align*}

As
\[P\left(c_{0}\delta f-P(c_{0}\delta f)-\psi|f|_{-}\right) = -\psi|f|_{-} < 0,\]
the series
\[\sum_{n=1}^{\infty}\sum_{y\in Fix_{n}}e^{c_{0}\delta f^{n}(y) -nP(c_{0}\delta f)-n\psi|f|_{-}}\]
converges to a constant $T<\infty$ (\cite{PP} cap 5).
It follows that
\begin{align*}
\limsup_{c(1-s)\rightarrow L} & \left(\frac{1}{c}\log \sum_{n=1}^{\infty}\sum_{y\in Fix_{n}}e^{csf^{n}-nP(cf)}\frac{k^{n}}{n}\right)\\
&\leq \limsup_{c(1-s)\rightarrow L} \frac{1}{c}\log\left( e^{-cI(1-\delta)}T \right)\\
&= -I(1-\delta).
\end{align*}
Now taking $\delta\rightarrow 0$, we get the upper bound inequality. \ \ \ \ \ \

\end{proof}
\bigskip
\section{The function $I$ and its extension $\widetilde{I}$}

\bigskip

For a periodic point $x$ we denote $n_{x}$ its minimum period.

Remember that for $x\in PER$ the function $I(x)$ is given by
\[I(x):=n_{x}\left(\beta(f) - \frac{f^{n_{x}}(x)}{n_{x}}\right) =n_{x}\beta(f) - f^{n_{x}}(x) .\]

We will show that  $I$  can be extended in a finite way to  a function $\widetilde{I}$  defined the all Bernoulli space $X$. This  $\widetilde{I}$ is non-negative, lower semi-continuous and such that the infimum of $I$ and $\widetilde{I}$ are the same in each cylinder set. This function $\widetilde{I}$ will be a deviation function for the family $\mu_{c,s}$ and will be different from the deviation function described in \cite{BLT} (which did not consider zeta measures).

By definition $\widetilde{I}:X\rightarrow \mathbb{R}\cup\{\infty\}$ is lower semi-continuous if  for any $x\in X$ and sequence $x_{m}\rightarrow x$ we have
\[\liminf_{m\rightarrow\infty}\widetilde{I}(x_{m})\geq \widetilde{I}(x).\]

\begin{definition}
We define $\widetilde{I}:X \rightarrow \mathbb{R}\cup\{\infty\}$ by
\[\widetilde{I}(x)= \lim_{\epsilon\rightarrow 0} \left(\inf \{I(y) : d(y,x) \leq \epsilon \}\right).\]
\end{definition}
As $I\geq 0$, and
\[\epsilon_{1} \leq \epsilon_{2} \Rightarrow \inf \{I(y) : d(y,x) \leq \epsilon_{1} \} \geq \inf \{I(y) : d(y,x) \leq \epsilon_{2} \}, \]
we have that $\widetilde{I}$ is well defined.
\begin{lemma} Suppose $x\in PER$ and $I(x)\neq 0$. Then
\[\lim_{\epsilon\rightarrow 0} \left(\inf \{I(y) : 0 < d(y,x) \leq \epsilon \}\right) = +\infty.\]
As a consequence we have:
\[I(x)=\widetilde{I}(x), \ \  x\in PER.\]
\end{lemma}

\begin{proof}
Suppose $x\in PER$ and $I(x)\neq 0$.
Let
\[Y_{j} := \{y \in PER : y\neq x \ \ and \ \ d(x,y)\leq \theta^{j}\}.\]
We only need to show that
\[\lim_{j\to \infty} \inf_{y \in Y_{j}} I(y) = +\infty.\]
Let
\[Y_{j}^{-}:=\{y\in Y_{j}: n_{y}\leq j\} \ \ and \ \ Y_{j}^{+}:=\{y\in Y_{j}: n_{y} > j\}.\]
We are going to show that
\[\lim_{j\to \infty} \inf_{y\in Y_{j}^{-}}I(y) = \lim_{j\to \infty} \inf_{y\in Y_{j}^{+}}I(y) =\infty.\]

Suppose first that $y\in Y_{j}^{-}$. By hypothesis $f \in F_{\theta}$, then there exists $C>0$ such that
\begin{align*}
f^{n_{y}}(y) &= f(y) + f(\sigma(y)) + f(\sigma^{2}(y)) +...+ f(\sigma^{n_{y}-1}(y))\\
&\leq (f(x) + C\theta^{j}) + (f(\sigma(x)) + C\theta^{j-1}) + ... + (f(\sigma^{n_{y}-1}(x)) + C\theta^{j-n_{y}+1})\\
&\leq f^{n_{y}}(x) +C\frac{\theta}{1-\theta}.
\end{align*}
We write $n_{y} = a(y)n_{x} + b(y), \ \ 0\leq b(y) < n_{x}$. Then
\begin{align*}
I(y) &= n_{y}\beta(f) - f^{n_{y}}(y) \geq n_{y}\beta(f) - f^{n_{y}}(x) - C\frac{\theta}{1-\theta}\\
&= (a(y)n_{x}+b(y))\beta(f) - f^{a(y)n_{x}+b(y)}(x) - C\frac{\theta}{1-\theta}\\
&= a(y)(n_{x}\beta(f) -f^{n_{x}}(x)) + b(y)\beta(f) -f^{b(y)}(x) - C\frac{\theta}{1-\theta}\\
&\geq a(y)I(x) -n_{x}|f|_{\infty} - C\frac{\theta}{1-\theta}\\
&= a(y)I(x) -C_{1},
\end{align*}
where $C_{1}$ not change with $y$ and $j$.
Then
\[ \lim_{j\to \infty} \inf_{y\in Y_{j}^{-}}I(y) \geq \lim_{j\to \infty} \inf_{y\in Y_{j}^{-}}a(y)I(x) -C_{1} = \infty,\]
because
\[\lim_{j\to \infty} \inf\{n_{y} : y\in Y_{j}^{-}\} = \infty.\]

\bigskip

Now suppose that $y\in Y_{j}^{+}$. Then $n_{y}=j+i, \ \ i > 0$, and we write
\[y := \overline{y_{1}...y_{j}y_{j+1}...y_{j+i}},\]
and define
\[z := \overline{y_{j+1}...y_{j+i}}.\]
Then,
\begin{align*}
f^{i}(\sigma^{j}(y)) &= f(\sigma^{j}(y)) + f(\sigma^{j+1}(y)) +...+ f(\sigma^{j+i-1}(y))\\
&\leq (f(z) + C\theta^{i}) + (f(\sigma(z)) + C\theta^{i-1}) + ... + (f(\sigma^{i-1}(z)) + C\theta)\\
&\leq f^{i}(z) + C\frac{\theta}{1-\theta},
\end{align*}
and, also
\begin{align*}
f^{j}(y) &= f(y) + f(\sigma(y)) + f(\sigma^{2}(y)) +...+ f(\sigma^{j-1}(y))\\
&\leq (f(x) + C\theta^{j}) + (f(\sigma(x)) + C\theta^{j-1}) + ... + (f(\sigma^{j-1}(x)) + C\theta)\\
&\leq f^{j}(x) +C\frac{\theta}{1-\theta}.
\end{align*}
So
\[f^{j+i}(y)=f^{j}(y)+f^{i}(\sigma^{j}(y))\leq f^{j}(x)+f^{i}(z)+ 2C\frac{\theta}{1-\theta}.\]
We write $j = a(j)n_{x} + b(j), \ \ 0\leq b(j) < n_{x}$. Then
\begin{align*}
I(y) &= (j+i)\beta(f) - f^{j+i}(y)\\
&\geq (j+i)\beta(f) - f^{j}(x) - f^{i}(z) - 2C\frac{\theta}{1-\theta}\\
&\geq i\beta(f)-f^{i}(z) +j\beta(f) -f^{j}(x) - 2C\frac{\theta}{1-\theta}\\
&\geq  I(z) + (a(j)n_{x}+b(j))\beta(f) -f^{a(j)n_{x}+b(j)}(x) - 2C\frac{\theta}{1-\theta}\\
&\geq  a(j)n_{x}\beta(f) +b(j)\beta(f) -a(j)f^{n_{x}}(x) - f^{b(j)}(x) - 2C\frac{\theta}{1-\theta}\\
&\geq  a(j)I(x)  -n_{x}|f|_{\infty} - 2C\frac{\theta}{1-\theta}\\
&=  a(j)I(x)  -C_{1},
\end{align*}
where $C_{1}$ not change with $y$ and $j$.
Then, finally
\[ \lim_{j\to \infty} \inf_{y\in Y_{j}^{+}}I(y) \geq \lim_{j\to \infty} a(t)I(x) -C_{1} = \infty.\]

\end{proof}

\begin{corollary}
Given $x\in PER$, then there is a cylinder $k$ such that $x \in k$ and $\inf_{y \in Per\cap k}I(y) = I(x)$.
\end{corollary}
\begin{proof}
If $I(x)=0$ there is nothing to prove.
\newline
If $I(x)\neq 0$, then we can use the lemma.
\end{proof}

\begin{corollary} Let $x\in PER$. Then, the following are equivalent:
\newline
i) $I(x)=0$
\newline
ii) $\mu_{x}$ given by $\mu_{x}(g)=\frac{g^{n_{x}}(x)}{n_{n}}$ is in $M_{max}(f)$
\newline
iii) $x \in supp(\mu_{\infty})$, for some $\mu_{\infty}\in M_{max}(f)$.
\end{corollary}
\begin{proof}
It is easy that
$i) \leftrightarrow ii)$, and $ii) \to iii)$. We are going to prove that:  not true  \, i)$ \,\to $  not true \, iii).
\newline
Suppose $I(x) \neq 0$.  By the corollary above there is a cylinder $k$, such that $x \in k$ and $\inf_{y \in Per\cap k}I(y) = I(x) \neq 0$.
We are going to prove that if $\mu_{\infty} \in M_{max}$, then $\mu_{\infty}(k) = 0$ (this show that $x \notin supp(\mu_{\infty})$).
We remark that we can suppose $\mu_{\infty} \in M_{max}$ is ergodic.
\newline
To prove that $\mu_{\infty}(k) = 0$, we have to  prove that if $\mu_{\infty}(k) \neq 0$, then $$\inf_{y \in Per\cap k}I(y) = I(x) = 0.$$ But, this is false.

We remark that if $\mu_{\infty}(k) \neq 0$, then the same ideas in proof of lemma \ref{lema}, that is $\inf_{x\in PER} I(x)=0$, can be used to prove that $\inf_{y \in Per\cap k}I(y) = I(x) = 0$.

\end{proof}

\begin{corollary}
Let $f\in F_{\theta}$. Suppose there is a unique $\mu_{\infty}$ in $M_{max}(f)$. Then $\mu_{\infty}$ has support in a periodic orbit, or there are no periodic points in the support of $\mu_{\infty}$.
\end{corollary}
\begin{proof}
If there is a $x \in PER$ such that $x \in supp(\mu_{\infty})$, then ( $iii) \to ii)$ ) $\mu_{x} \in  M_{max}(f)$, so $\mu_{\infty}=\mu_{x}$.
\end{proof}

\begin{lemma}
 The function $\widetilde{I}:X \rightarrow \mathbb{R}\cup\{\infty\}$ is non-negative, lower semi-continuous and for all cylinder $k$
\[\inf_{k \cap X}\widetilde{I} = \inf_{k\cap PER}I.\]
\end{lemma}

\begin{proof}
It is trivial that $\widetilde{I} \geq 0$.
Suppose $x$ and $\{x_{m}\}$ in $X$ are such that  $x_{m}\rightarrow x$. If $\widetilde{I}(x)=0$ there is nothing to prove. Suppose $\widetilde{I}(x)>0$. Take $\delta>0$ such that $\widetilde{I}(x)>\delta$. By definition of $\widetilde{I}(x)$, there exists $\epsilon>0$ such that for all $y\in \text{PER}$ with $d(x,y)<\epsilon$, we have that $I(y)>\delta$. If $m$ is large enough $d(x_{m},x)<\epsilon/2$. It follows that for large $m$
\[ \inf \{I(y) : d(y,x_{m}) \leq \epsilon/2 \} \geq \inf \{I(y) : d(y,x) \leq \epsilon \} \geq \delta.\]
Therefore, $\widetilde{I}(x_{m})\geq \delta$, and finally
\[\liminf_{m\rightarrow\infty} \widetilde{I}(x_{m}) \geq \delta.\]
As we take any  $\delta < \widetilde{I}(x)$, we have that
\[\liminf_{m\rightarrow\infty} \widetilde{I}(x_{m}) \geq \widetilde{I}(x),\]
and this shows that $\widetilde{I}$ is lower semi-continuous.
Now, for a fixed cylinder
$k$, we will show that:
\[\inf_{k \cap X}\widetilde{I} = \inf_{k\cap \text{PER}}I.\]
We know that for any $y \in k\cap PER$
\[\widetilde{I}(y) = I(y),\]
then
\[\inf_{k \cap X}\widetilde{I} \leq \inf_{k \cap PER}\widetilde{I}  = \inf_{k\cap PER}I.\]
We have to show that
\[\inf_{k \cap X}\widetilde{I} \geq \inf_{k\cap PER}I.\]
 Consider $x_{m}$ a sequence of elements in $k\cap X$ such that $\widetilde{I}(x_{m})\rightarrow \inf_{k\cap X}\widetilde{I}$. Denote by $x \in k \cap X$ an accumulation point of $\{x_{m}\}$. Then, as $\widetilde{I}$ is lower semi-continuous
\[\widetilde{I}(x)\leq \liminf_{m\rightarrow\infty} \widetilde{I}(x_{m}),\]
that is,
\[\widetilde{I}(x)=\inf_{k \cap X}\widetilde{I}.\]

From the definition $\widetilde{I}(x)$, there exists $\{y_{m}\}$ in $k\cap \text{PER}$ such that $y_{m}\rightarrow x$ e $I(y_{m})\rightarrow \widetilde{I}(x)$. It follows that
\[ \inf_{k \cap X}\widetilde{I} = \widetilde{I}(x) \geq \inf_{k\cap \text{PER}}I.\]
\end{proof}

From this lemma  and theorem \ref{ldt} it follows that

\begin{corollary}

The probabilities $\mu_{c,s}$ satisfies a Large Deviation Principle with deviation function $\widetilde{I}$:
for fixed $L>0$, and for any cylinder  $k\subset X$
\[\lim_{c(1-s)\rightarrow L}\frac{1}{c}\log \mu_{c,s}(k) = - \inf_{x\in k}\widetilde{I}(x),\]
where $\widetilde{I}$ is lower semi-continuous and non-negative.
The same is true if we have:
\[\liminf_{c\to\infty, \, s\to1}c(1-s)=L>0.\]
\end{corollary}
The equilibrium measures $\mu_{c\,f}$ for $cf$ converge to $\mu_\infty$ (when $\mu_\infty \in M_{max}(f)$ is unique). According to \cite{BLT} they satisfy a L. D. P. with deviation function $I_{BLT}$:

 That is, when $\mu_\infty \in M_{max}(f)$ is unique, for any cylinder $k\subset X=\{0,1\}^{\mathbb{N}}$
\[\lim_{c\to\infty} \frac{1}{c}\log(\mu_{c\,f}(k))=-\inf_{x\in k}I_{BLT}(x).\]

The deviation function $ I_{BLT}$ is non-negative, lower semi-continuous but is finite only in the pre-images of points in the support of the maximizing probability.

We will show that:

\begin{proposition} Suppose  $\mu_\infty$ is the unique maximizing probability for $f$ as above.
Then, there exists a cylinder $k$ such that
\[\inf_{x\in k}\widetilde{I} \neq \inf_{x\in k} I_{BLT}.\]
\end{proposition}
\begin{proof}

 We fix a periodic point $x$ such that $ I_{BLT}(x) = \infty$, and for each  $m$ we consider the cylinder $k_{m}=[x_{1}...x_{m}]$.
We know that
\[\widetilde{I}(x) = I(x) < \infty \ \ \ \text{and} \ \ \ \ I_{BLT}(x) = \infty.\]
As, for each $m$
\[\inf_{k_{m} \cap X}\widetilde{I} \leq \widetilde{I}(x),\]
we just have to show that:
\newline
\textbf{Claim:} there exists $m$ such that $\inf_{k_{m} \cap X}\,I_{BLT} > \widetilde{I}(x)$.
\newline

 The proof will be by contradiction. Suppose that for any $m$, we have that $\inf_{k_{m} \cap X}I_{BLT} \leq \widetilde{I}(x)$. Then, for each $m$ denote $x_{m} \in (k_{m} \cap X)$ a point which realizes $\inf_{k_{m}\cap X}I_{BLT}$. Therefore, we get a  sequence $x_{m}\rightarrow x$, such that
\[\liminf_{m \rightarrow \infty} I_{BLT}(x_{m}) \leq \widetilde{I}(x) < I_{BLT} (x),\]
and this is in contradiction with the fact that $I_{BLT}$ is lower semi-continuous \cite{BLT}.
\end{proof}

\section{the case $c(1-s) \rightarrow 0$}
In the previous sections, under the condition $c(1-s)\rightarrow L>0$, we get a L. D. P. with deviation function $\widetilde{I} \neq I_{BLT}$.
\newline

This raises the question: what happens if $c(1-s)\rightarrow 0$?
We will show here that  the final conclusion is quite different in this case. We have that:
\begin{lemma}\label{sto1}
For each continuous function $k:X\to \mathbb{R}$
\[\lim_{s \to 1} \mu_{c,s}(k) = \mu_{cf}(k),\]
where $\mu_{cf}$ is the invariant equilibrium state for $cf$
\end{lemma}
From this lemma, it is not surprise that:

\begin{proposition}\label{zero}
Suppose that $\mu_\infty \in M_{max}(f)$ is unique and $X=\{0,1\}^{\mathbb{N}}$.
If $c(1-s)\rightarrow 0$ fast enough, then for all cylinder $k$
\[\lim_{c(1-s)\rightarrow 0 } \,\,\frac{1}{c}\log (\mu_{c,s}(k)) = \lim_{c\to\infty} \frac{1}{c}\log(\mu_{c\,f}(k)) = -\inf_{x\in k}I_{BLT}(x),\]
where $\mu_{cf}$ is the invariant equilibrium state for $cf$.
\end{proposition}

We will need some results presented in \cite{PP}.

\begin{lemma}
Consider $g_{0}$ a real Lipschitz potential in $F_{\theta}$.
\newline
a)If $P(g_{0})<0$, then, for $g$ close by $g_{0}$ (in the Lipschitz norm)
\[\sum_{n=1}^{\infty}\frac{1}{n}\,|\sum_{\text{Fix}_{n}}e^{g^{n}(x)}\,| < \infty.\]
b) If $P(g_{0})=0$  then, for $g$ close by $g_{0}$ (in the Lipschitz norm)
\[\sum_{n=1}^{\infty}\frac{1}{n}\left|\sum_{x\in \text{Fix}_{n}}e^{g^{n}(x)}-e^{nP(g)}\right| < \infty.\]
\end{lemma}

\begin{proof}
For  a) see page 80, Theo. 5.4 \cite{PP} and
 for b) see  page 81 Theo. 5.5 (ii) \cite{PP}. Note that for a real  $g$ we have the spectral radius $\rho(L_{g})=e^{P(g)}.$
\end{proof}

 Note that for $s \in (0,1)$ we have $P(csf-P(cf))<0$, therefore, for any $k\in F_{\theta}$ fixed
\[(s,z) \rightarrow \sum_{n=1}^{\infty}\frac{1}{n}\sum_{Fix_{n}}e^{csf^{n}+zk^{n}-P(cf)n},\]
is analytic for $s \in (0,1)$ and $|z|$ small (in a small neighborhood that depends of $s$ and $c$). When convenient $z$ will be real.
\newline
From this the function
\[\zeta(s,z)=\exp\left( \sum_{n=1}^{\infty}\frac{1}{n}\sum_{Fix_{n}}e^{csf^{n}+zk^{n}-P(cf)n} \right),\]
 is not zero at $z=0$, and is analytic for $s \in (0,1)$ and $|z|$ small (in this small neighborhood that depends on  $s$ and $c$). Therefore:

\begin{proposition}
If we denote the partial derivative of  $\zeta$ in the variable $z$ by $\zeta_{2}$, then
\[\frac{\zeta_{2}(s,0)}{\zeta(s,0)} = \sum_{n=1}^{\infty}\sum_{x\in Fix_{n}}e^{csf^{n}-nP(cf)}\frac{k^{n}}{n}\]
is analytic for $s\in(0,1)$.
\end{proposition}

Moreover:
\begin{proposition}
For each real value $c$:
\newline
i) the function
\[\alpha(s,z) = \exp \left( \sum_{n=1}^{\infty}\frac{1}{n}\left[\left(\sum_{Fix_{n}}e^{csf^{n}+zk^{n}-P(cf)n}\right)-e^{nP(csf+zk-P(cf))}\right]\right),\]
is analytic for $s\in (0,1]$ and $z$ in a small neighborhood that depends on $s$ and $c$.
\newline
ii) For $s\in(0,1)$ and  $z$ in  a small neighborhood that depends on $s$ and $c$
\[\alpha(s,z)=\zeta(s,z)(1-e^{P(csf+zk-P(cf))}).\]
\end{proposition}

\begin{proof}
 For ii) we just have to  use  \[\sum_{n=1}^{\infty}\frac{1}{n}z^{n}= -\log(1-z),|z|<1\]
 Therefore, for $s\in(0,1)$ we have $P(csf+zk-P(cf))<0$, for $z$   in a small neighborhood that depends on $s$ and $c$. In particular,  $e^{P(csf+zk-P(cf))}<1$ and we can write
\[\log(1-e^{P(csf+zk-P(cf))})=-\sum_{n=1}^{\infty}\frac{1}{n}e^{nP(csf+zk-P(cf))}.\]
It follows that:

\begin{align*}
\alpha(s,z) &= \exp \left( \sum_{n=1}^{\infty}\frac{1}{n}\left[\left(\sum_{FIX_{n}}e^{csf^{n}+zk^{n}-P(cf)n}\right)-e^{nP(csf+zk-P(cf))}\right]\right)\\
&= \exp\left( \sum_{n=1}^{\infty}\frac{1}{n}\sum_{Fix_{n}}e^{csf^{n}+zk^{n}-P(cf)n} \right) \exp\left(-\sum_{n=1}^{\infty}\frac{1}{n}e^{nP(csf+zk-P(cf))}\right)\\
&= \zeta(s,z)\,e^{\log(1-e^{P(csf+zk-P(cf))}}\\
&= \zeta(s,z)\,(1-e^{P(csf+zk-P(cf))}).
\end{align*}
 Now we prove i).
\newline
When  $s\in(0,1)$, we just have to use ii) and the fact that $P$ is analytic.
When $s=1$ we have to use the previous lemma (b) and the fact that $\exp$ is analytic.
\end{proof}

As $\alpha(s,0)\neq 0$, we can calculate $\frac{\alpha_{2}(s,0)}{\alpha(s,0)}$. Then we get:

\begin{lemma}\label{alpha}
The function $\frac{\alpha_{2}(s,0)}{\alpha(s,0)}$ is analytic for  $s\in(0,1]$ and $z$ in a small neighborhood that depends on $s$ and $c$.
\newline
For $s\in(0,1)$
\[\frac{\alpha_{2}(s,0)}{\alpha(s,0)} =\frac{\zeta_{2}(s,0)}{\zeta(s,0)} -\frac{e^{P(csf)-P(cf)}}{1-e^{P(csf)-P(cf)}}\int k d\mu_{csf}.\]
\end{lemma}

\begin{proof}
We remark that
$\left. \frac{\partial P(csf+zk)}{\partial z}\right| _{z=0}=\int k d\mu_{csf}$ (see \cite{PP} page 60).
\newline
For $s\in(0,1)$ we have \[\alpha(s,z)=\zeta(s,z)(1-e^{P(csf+zk-P(cf))}),\]
then
\begin{align*}
\frac{\alpha_{2}(s,0)}{\alpha(s,0)} &= \frac{\zeta_{2}(s,0)(1-e^{P(csf)-P(cf)})-\zeta(s,0)e^{P(csf)-P(cf)}\left. \frac{\partial P(csf+zk)}{\partial z}\right| _{z=0}}{\zeta(s,0)(1-e^{P(csf)-P(cf)})}\\
&= \frac{\zeta_{2}(s,0)}{\zeta(s,0)} - \frac{e^{P(csf)-P(cf)}}{1-e^{P(csf)-P(cf)}}\left. \frac{\partial P(csf+zk)}{\partial z}\right| _{z=0}\\
&= \frac{\zeta_{2}(s,0)}{\zeta(s,0)} -\frac{e^{P(csf)-P(cf)}}{1-e^{P(csf)-P(cf)}}\int k d\mu_{csf}.
\end{align*}
\end{proof}
Now we will show the proof of Lemma \ref{sto1}

\begin{proof}
Fix a cylinder $k$. We know that
\begin{align*}
\sum_{n=1}^{\infty}\sum_{x\in Fix_{n}}e^{csf^{n}-nP(cf)}\frac{k^{n}}{n} &=
\frac{\zeta_{2}(s,0)}{\zeta(s,0)}\\
&=  \frac{\alpha_{2}(s,0)}{\alpha(s,0)} + \frac{e^{P(csf)-P(cf)}}{1-e^{P(csf)-P(cf)}}\int k d\mu_{csf}.
\end{align*}
This means
\begin{align*}
\frac{1-e^{P(csf)-P(cf)}}{e^{P(csf)-P(cf)}}\sum_{n=1}^{\infty}\sum_{x\in Fix_{n}}e^{csf^{n}-nP(cf)}\frac{k^{n}}{n} =  \frac{1-e^{P(csf)-P(cf)}}{e^{P(csf)-P(cf)}}\frac{\alpha_{2}(s,0)}{\alpha(s,0)} + \int k d\mu_{csf}.
\end{align*}
We can also consider  the same reasoning for $k \equiv 1$. Now, taking the quotient we get
\begin{equation}
\mu_{c,s}(k) =  \huge \frac{\frac{1-e^{P(csf)-P(cf)}}{e^{P(csf)-P(cf)}}\frac{\alpha_{2}(s,0)}{\alpha(s,0)} + \int k d\mu_{csf}}{\frac{1-e^{P(csf)-P(cf)}}{e^{P(csf)-P(cf)}}\frac{\beta_{2}(s,0)}{\beta(s,0)} + 1} \normalsize,
\label{3}
\end{equation}
where $\beta$ represents the function $\alpha$ when $k\equiv 1$.

It is known that for fixed $c$, the value $\int k d\mu_{csf}$ depends in a continuous way on $s$ (it's the derivative in $z=0$ of $P(c\,s\,f + z\,k$).
Then when we take $s \to 1$ on $(\ref{3})$ we have
\[\lim_{s\to1}\mu_{c,s}(k)=\mu_{cf}(k).\]
Now, when $g:X\to\mathbb{R}$ is continuous, we approximate  by functions that can be written as linear combinations of characteristic functions of   cylinders and repeat the argument on proof of Theorem \ref{teo}.
So
\[g \to \lim_{s\to1}\mu_{c,s}(g)\]
is a measure and have the same value that $\mu_{cf}$ on cylinders. Then
\[\lim_{s\to1}\mu_{c,s}(g)=\mu_{cf}(g).\]
\end{proof}

Now we will show the proof of Proposition \ref{zero}

\begin{proof}
We just have to investigate cylinders $k$ such that $\mu_\infty(k)=0$. Consider a enumeration $k_{1},k_{2},...$ of all cylinders such that $\mu_\infty(k)=0$. We begin with a fixed $k_{i}$ with this property and denote this by $k$.

For fixed $c$, the value $\int k d\mu_{csf}$ depends in a continuous way on $s$. In particular, as $\mu_{csf}$ are Gibbs states and therefore positive in open sets, then $\int k d\mu_{c\,s\,f} > 0$ for each $s$. It follows that
\[A_{c} = \inf_{s\in[1/2,1]}\int k d\mu_{c\,s\,f} > 0.\]
\newline
As for each fixed $c$, (with the notation of the proof above)
\[\frac{\beta_{2}(s,0)}{\beta(s,0)}\ \ \ \ \text{and} \ \ \ \ \frac{\alpha_{2}(s,0)}{\alpha(s,0)}\]
are analytic on  $s=1$, and, as for each fixed $c$
\[\frac{1-e^{P(csf)-P(cf)}}{e^{P(csf)-P(cf)}}\stackrel{s\rightarrow 1}{\rightarrow} 0,\]
then, we can find for each $c$, a value $s_{c}^{i}$ (remark that $i$ is the index of the cylinder $k=k_{i}$ fixed) such that $c(1-s_{c}^{i})\rightarrow 0$, and if $s_{c}^{i}<s<1$:
\begin{align*}
& a) -1/2 < \frac{1-e^{P(csf)-P(cf)}}{e^{P(csf)-P(cf)}}\frac{\beta_{2}(s,0)}{\beta(s,0)} < 1/2;\\
& \\
& b) - \frac{\int k d\mu_{csf}}{2} \leq -A_{c}/2 < \frac{1-e^{P(csf)-P(cf)}}{e^{P(csf)-P(cf)}}\frac{\alpha_{2}(s,0)}{\alpha(s,0)} < A_{c}/2 \leq \frac{\int k d\mu_{csf}}{2}.
\end{align*}

It follows from $(\ref{3})$ above that for each $c$ and $s_{c}^{i}<s<1$:
\[\frac{-\frac{\int k d\mu_{csf}}{2} + \int k d\mu_{csf}}{3/2} \leq \mu_{c,s}(k) \leq \frac{\frac{\int k d\mu_{csf}}{2} + \int k d\mu_{csf}}{1/2}.\]
This means
\[\frac{\int k d\mu_{csf}}{3} \leq \mu_{c,s}(k) \leq 3\int k d\mu_{csf}.\]
The conclusion is
\begin{align*}
\lim_{(c\to\infty,\, s_{c}^{i}<s<1)}\frac{1}{c}\log(\mu_{c,s}(k)) &= \lim_{(c\rightarrow\infty,\, s_{c}^{i}<s<1)}\frac{1}{sc}\log(\mu_{c,s}(k))\\
&= \lim_{(c\rightarrow\infty, \, s_{c}^{i}<s<1)}\frac{1}{sc}\log(\mu_{csf}(k))\\
&= -\inf_{x\in k}I_{BLT}(x).
\end{align*}
We remember that in this argument $k$ is fixed. For each $k_{i}$ we have a association $c\to s_{c}^{i}$ described above. Consider now the association $c\to s_{c}$, where for each integer $n>0$:
\[c \in [n,n+1) \Rightarrow s_{c} = sup_{i\in\{1,...,n\}}s_{c}^{i} < 1.\]
Then for each cylinder $k_{i}$ we have that $c > i \Rightarrow s_{c}^{i}<s_{c}$, so has above
\[ \lim_{c\rightarrow\infty,\, s_{c}<s<1}\frac{1}{c}\log(\mu_{c,s}(k)) = \lim_{c\rightarrow\infty,\, s_{c}^{j}<s<1}\frac{1}{c}\log(\mu_{c,s}(k))
=  -\inf_{x\in k}I_{BLT}(x).\]

\end{proof}

\section{ About the measures $\pi_{c,N}$ and $\eta_{c,N}$ }

There are other ways to approximate $\mu_{cf}$ and $\mu_{\infty}$: we can use, for example,  the measures $\pi_{c,N}$ and $\eta_{c,N}$, $c \in \mathbb{R}$ and $N\in \mathbb{N}$, given by
\[\pi_{c,N}(k) =
 \frac{\sum_{n=1}^{N}\sum_{x\in Fix_{n}}e^{c\,f^{n}(x)-n\,P(c\,f)}\frac{k^{n}(x)}{n}}{\sum_{n=1}^{N}\sum_{x\in Fix_{n}}e^{c\,f^{n}(x)-n\,P(c\,f)}}\]
and
\[\eta_{c,N}(k) =  \frac{\sum_{n=1}^{N}\sum_{x\in Fix_{n}}e^{c\,f^{n}(x)}\frac{k^{n}(x)}{n}}{\sum_{n=1}^{N}\sum_{x\in Fix_{n}}e^{c\,f^{n}(x)}}.\]

We will prove:

\begin{lemma}\label{lemma2}
For fixed $c$ and continuous $g:X \to \mathbb{R}$:
\[\lim_{N \to \infty} \pi_{c,N}(g) = \lim_{N\to \infty}\eta_{c,N}(g) = \mu_{c\,f}(g),\]
where $\mu_{c\,f}$ is the Gibbs state for $cf$.
\end{lemma}

After that we will analyze what happens when $c \to \infty$, simultaneously, with $N \to \infty$. We prove that:

\begin{theorem}\label{teo2}

When $c,N\rightarrow\infty$ any accumulation point of $\pi_{c,N}$ (or $\eta_{c,N}$) is in $M_{max}(f)$.
Moreover, if $g:X\rightarrow \mathbb{R}$ is a continuous function, $c_{j}\rightarrow\infty$ and $N_{j}\rightarrow \infty$, are such that there exist
\[\lim_{j\rightarrow\infty}\pi_{c_{j},N_{j}}(g), \]
then, this limit is $\int g d\mu$ for some accumulation point $\mu$ of  $\pi_{c,N}$ in the weak* topology. (the same happens for $\eta_{c,N}$).
\newline
\newline
In particular, if $\mu_\infty$ is unique in $M_{max}(f)$, then for any continuous  $g:X\rightarrow \mathbb{R}$
\[\lim_{c,N \to \infty} \pi_{c,N}(g) = \lim_{c,N\to \infty}\eta_{c,N}(g)=\int g d\mu_\infty.\]
\end{theorem}

We also study this result for $\pi_{c,N}$ in order to get a L.D.P..

\begin{theorem}\label{ldt2}
Suppose $f$ is Lipschitz. Then for all cylinder $k\subset X$
\[\lim_{\frac{N}{c} \to 0} \frac{1}{c}\log(\pi_{c,N}(k))=-\inf_{x\in k, \,x \in \text{Per}}I(x) = - \inf_{x \in k}\widetilde{I}(x).\]
\end{theorem}

We point out  that we take $c,N \to \infty$.
\newline
\newline
\textbf{We start with the proof of Lemma \ref{lemma2}:}

\begin{proof}
By Lemma \ref{alpha} the function

\[ \frac{\alpha_{2}(s,0)}{\alpha(s,0)} =\sum_{n=1}^{\infty}\left[ \left(\sum_{Fix_{n}}e^{csf^{n}-nP(cf)}\frac{k^{n}}{n}\right) - e^{nP(csf)-nP(cf)}\int k d\mu_{csf} \right]  \]
is analytic on $s=1$.

Then:
\[-\infty <\frac{\alpha_{2}(1,0)}{\alpha(1,0)} = \sum_{n=1}^{\infty}\left[\left(\sum_{Fix_{n}}e^{cf^{n}-nP(cf)}\frac{k^{n}}{n}\right)-\int kd\mu_{cf}\right] < \infty.\]
So
\begin{equation}
\left(\sum_{Fix_{n}}e^{cf^{n}-nP(cf)}\frac{k^{n}}{n}\right) \stackrel{n\to\infty}{\to} \int kd\mu_{cf}. \label{5}
\end{equation}
Then,
\[\frac{\sum_{Fix_{n}}e^{cf^{n}-nP(cf)}\frac{k^{n}}{n}}{\sum_{Fix_{n}}e^{cf^{n}-nP(cf)}}\stackrel{n\to\infty}{\to} \int kd\mu_{cf},\]
and
\[\frac{\sum_{Fix_{n}}e^{cf^{n}}\frac{k^{n}}{n}}{\sum_{Fix_{n}}e^{cf^{n}}}\stackrel{n\to\infty}{\to} \int kd\mu_{cf}.\]
Now we use the following well known result:

\bigskip

\begin{center}
``Let $a_{n}$ and $b_{n}$ sequences of real numbers $\geq0$. Suppose that $\frac{a_{n}}{b_{n}} \to L$, and there is $\epsilon>0$ such that $a_{n}>\epsilon$ and $b_{n}>\epsilon$, $n>>0$ .
\newline
Then $\frac{\sum_{n=1}^{N}a_{n}}{\sum_{n=1}^{N}b_{n}}$ goes to $L$ when $N\rightarrow\infty$."
\end{center}

\bigskip

We have then (using (\ref{5}) in order to get the $> \epsilon$ property) that for $k>0$ (or, cylinder sets):
\[\lim_{N\to\infty}\pi_{c,N}(k) = \lim_{N\to\infty}\frac{\sum_{n=1}^{N}\sum_{Fix_{n}}e^{cf^{n}-nP(cf)}\frac{k^{n}}{n}}{\sum_{n=1}^{N}\sum_{Fix_{n}}e^{cf^{n}-nP(cf)}}=\int kd\mu_{cf},\]
and
\[\lim_{N\to\infty}\eta_{c,N}(k) = \lim_{N\to\infty}\frac{\sum_{n=1}^{N}\sum_{Fix_{n}}e^{cf^{n}}\frac{k^{n}}{n}}{\sum_{n=1}^{N}\sum_{Fix_{n}}e^{cf^{n}}}=\int kd\mu_{cf}.\]
When $g:X \to \mathbb{R}$ is continuous, we use aproximation arguments.
\end{proof}

\bigskip

\textbf{Now we prove the Theorem \ref{teo2}:}

\begin{proof}
We start with $\pi_{c,N}$. Using the same arguments that we used in Theorem \ref{teo} we only need prove that
\[\liminf_{c,N\to\infty}\pi_{c,N}(f) \geq \beta(f).\]
We will use the same notations and ideas that Lemma \ref{aux}, so for fixed $\epsilon>0$ we only need prove that:
\[\frac{\sum_{n=1}^{N}\sum_{x\in A_{n}}e^{cf^{n}-nP(cf)}}{\sum_{n=1}^{N}\sum_{x\in B_{n}}e^{cf^{n}-nP(cf)}} \stackrel{c,N \to \infty}{\to} 0.\]
Now:
\begin{align*}
\sum_{n=1}^{N}\sum_{x\in A_{n}}e^{cf^{n}-nP(cf)} &\leq \sum_{n=1}^{N}\sum_{x\in A_{n}}e^{cn(\beta(f) - \epsilon)-nc\beta(f) -n\epsilon_{c}}\\
&=  \sum_{n=1}^{N}\sum_{x\in A_{n}}e^{- nc\epsilon -n\epsilon_{c}}\\
&\leq  \sum_{n=1}^{N}e^{- nc\epsilon + n\log(d)}\\
&\leq \frac{e^{- c\epsilon + \log(d)}}{1-e^{- c\epsilon + \log(d)}}.
\end{align*}

 On the other hand, by lemma $\ref{lema}$, there exists a periodic point $x$ such that:
$$I(x)=n_{x}\left(\beta(f) - \frac{f^{n_{x}}(x)}{n_{x}}\right)<\epsilon/2.$$
Therefore,
\begin{align*}
\sum_{n=1}^{N}\sum_{x\in B_{n}}e^{cf^{n}-nP(cf)} &\geq e^{cf^{n_{x}}(x)-n_{x}P(cf)}\\
&= e^{-cI(x) -n_{x}\epsilon_{c}} \\
&\geq e^{-c\epsilon/2 -n_{x}\epsilon_{c}}.
\end{align*}

It follows that
\[
\frac{\sum_{n=1}^{N}\sum_{x\in A_{n}}e^{cf^{n}-nP(cf)}}{\sum_{n=1}^{N}\sum_{x\in B_{n}}e^{cf^{n}-nP(cf)}}
\leq \frac{e^{- c\epsilon + \log(d)}}{e^{-c\epsilon/2 -n_{x}\epsilon_{c}}}\frac{1}{1-e^{- c\epsilon + \log(d)}} \to 0. \]
\bigskip
\bigskip

Now we consider $\eta_{c,N}$:
\newline
In the same way as above, using the same notations and ideas of Lemma \ref{aux}, we only need prove that for $\epsilon>0$ fixed:
\[\frac{\sum_{n=1}^{N}\sum_{x\in A_{n}}e^{cf^{n}}}{\sum_{n=1}^{N}\sum_{x\in B_{n}}e^{cf^{n}}} \stackrel{c,N \to \infty}{\to} 0.\]

We have
\begin{align*}
\sum_{n=1}^{N}\sum_{x\in A_{n}}e^{cf^{n}} &\leq \sum_{n=1}^{N}\sum_{x\in A_{n}}e^{cn(\beta(f) - \epsilon)}\\
&\leq  \sum_{n=1}^{N}e^{cn(\beta(f) - \epsilon)+n\log(d)}\\
&= e^{c(\beta(f) - \epsilon)+\log(d)}\frac{e^{cN(\beta(f) - \epsilon)+N\log(d)}-1}{e^{c(\beta(f) - \epsilon)+\log(d)}-1}.
\end{align*}

By the other side, there exists a periodic point $x$ such that:
$$ \frac{f^{n_{x}}(x)}{n_{x}} > \beta(f) -\epsilon/2.$$
Therefore,
\begin{align*}
\sum_{n=1}^{N}\sum_{x\in B_{n}}e^{cf^{n}} &\geq \sum_{j=1}^{[N/n_{x}]}e^{cjf^{n_{x}}(x)}\\
&\geq \sum_{j=1}^{[N/n_{x}]}e^{cjn_{x}(\beta(f)-\epsilon/2)}\\
&\geq e^{cn_{x}(\beta(f)-\epsilon/2)}\frac{e^{c[N/n_{x}]n_{x}(\beta(f)-\epsilon/2)}-1}{e^{cn_{x}(\beta(f)-\epsilon/2)}-1}.
\end{align*}
It follows that
\begin{align*}
\frac{\sum_{n=1}^{N}\sum_{x\in A_{n}}e^{cf^{n}}}{\sum_{n=1}^{N}\sum_{x\in B_{n}}e^{cf^{n}}}
&\leq \frac{e^{c(\beta(f) - \epsilon)+\log(d)}\frac{e^{cN(\beta(f) - \epsilon)+N\log(d)}-1}{e^{c(\beta(f) - \epsilon)+\log(d)}-1}}{e^{cn_{x}(\beta(f)-\epsilon/2)}\frac{e^{c[N/n_{x}]n_{x}(\beta(f)-\epsilon/2)}-1}{e^{cn_{x}(\beta(f)-\epsilon/2)}-1}}\\
&= \frac{e^{cN(\beta(f) - \epsilon)+N\log(d)}-1}{e^{c[N/n_{x}]n_{x}(\beta(f)-\epsilon/2)}-1}\frac{e^{c(\beta(f) - \epsilon)+\log(d)}}{e^{c(\beta(f) - \epsilon)+\log(d)}-1}\frac{e^{cn_{x}(\beta(f)-\epsilon/2)}-1}{e^{cn_{x}(\beta(f)-\epsilon/2)}}.
\end{align*}
Now, we have:
\[\lim_{c,N\to\infty}\frac{e^{cn_{x}(\beta(f)-\epsilon/2)}-1}{e^{cn_{x}(\beta(f)-\epsilon/2)}}=1 \ \ \ \ \text{and} \ \ \ \
\lim_{c,N\to\infty}\frac{e^{c(\beta(f) - \epsilon)+\log(d)}}{e^{c(\beta(f) - \epsilon)+\log(d)}-1} = 1.\]
By the other side,
\[\lim_{c,N\to\infty}\frac{e^{cN(\beta(f) - \epsilon)+N\log(d)}-1}{e^{cN(\beta(f) - \epsilon)}}=1 \ \ \ \ \text{and} \ \ \ \
\lim_{c,N\to\infty}\frac{e^{c[N/n_{x}]n_{x}(\beta(f)-\epsilon/2)}-1}{e^{cN(\beta(f)-\epsilon/2)}}=1.\]
Then
\[\lim_{c,N\to\infty}\frac{e^{cN(\beta(f) - \epsilon)+N\log(d)}-1}{e^{c[N/n_{x}]n_{x}(\beta(f)-\epsilon/2)}-1} = \lim_{c,N\to\infty}\frac{e^{cN(\beta(f) - \epsilon)}}{e^{cN(\beta(f)-\epsilon/2)}}=0.\]
So we have
\[\frac{\sum_{n=1}^{N}\sum_{x\in A_{n}}e^{cf^{n}}}{\sum_{n=1}^{N}\sum_{x\in B_{n}}e^{cf^{n}}} \stackrel{c,N \to \infty}{\to} 0 .\]

\bigskip
\bigskip

Now, given $g$, $c_{j}\to\infty$ and $N_{j}\to\infty$, such that exist $\lim_{j\to\infty}\pi_{c_{j},N_{j}}(g)$, we repeat the proof for $\mu_{c,s}$ and obtain an accumulation point $\pi_{\infty}$ such that
\[\lim_{j\to\infty}\pi_{c_{j},N_{j}}(g) = \pi_{\infty}(g).\]

The same is true for $\eta_{c,N}$.

\end{proof}

\textbf{Now we prove the Theorem \ref{ldt2}}

\begin{proof}

We just repeat the ideas used in the proof of Theorem \ref{ldt}.  We only need to prove that
\[\lim_{N/c\to 0} \frac{1}{c}\log\left( \sum_{n=1}^{N}\sum_{x\in Fix_{n}}e^{cf^{n}-nP(cf)}\frac{k^{n}}{n}   \right) = -\inf_{x\in k, \,x \in \text{Per}}I(x).\]

First we will show the lower inequality:
\[\liminf_{N/c\to 0} \frac{1}{c}\log\left( \sum_{n=1}^{N}\sum_{x\in Fix_{n}}e^{cf^{n}-nP(cf)}\frac{k^{n}}{n}   \right) \geq -\inf_{x\in k, \,x \in \text{Per}}I(x).\]
\newline
Consider a generic  point  $x\in k$ which is part of a periodic orbit $\{x,...,\sigma^{n_{x}-1}x\}$. For $N>>1$ we use that:
\begin{align*}
\sum_{n=1}^{N}\sum_{x\in Fix_{n}}e^{cf^{n}-nP(cf)}\frac{k^{n}}{n} &\geq
\sum_{\{x,...,\sigma^{(n_{x}-1)}x\}}e^{cf^{n_{x}}-n_{x}P(cf)}\frac{k^{n_{x}}}{n_{x}}\\
&= e^{cf^{n_{x}}(x)-n_{x}P(cf)}k^{n_{x}}(x)\\
&\geq  e^{cf^{n_{x}}(x)-n_{x}P(cf)}\\
&= e^{-cI(x)-n_{x}\epsilon_{c}}.
\end{align*}

From this follows that
\[
\liminf_{c,N \to \infty} \frac{1}{c}\log \left( \sum_{n=1}^{N}\sum_{x\in Fix_{n}}e^{cf^{n}-nP(cf)}\frac{k^{n}}{n}\right)\geq -I(x).\]

\bigskip

Now we will show the upper inequality
\[\limsup_{N/c \to 0} \frac{1}{c}\log \left( \sum_{n=1}^{N}\sum_{x\in Fix_{n}}e^{cf^{n}-nP(cf)}\frac{k^{n}}{n}\right)\leq-\inf_{x\in k, \, x \in \text{PER}}I(x).\]
We will denote the value $\inf_{x\in k, \, x \in \text{PER}}I(x)$ by $I$. Then:
\begin{align*}
\sum_{n=1}^{N}\sum_{x\in Fix_{n}}e^{cf^{n}-nP(cf)}\frac{k^{n}}{n} &\leq \sum_{n=1}^{N}\sum_{x\in Fix_{n}}e^{-cn\left(\beta(f)-\frac{f^{n}}{n}\right)-n\epsilon_{c}}\frac{k^{n}}{n}\\
&\leq \sum_{n=1}^{N}\sum_{x\in Fix_{n}}e^{-cI -n\epsilon_{c}}\frac{k^{n}}{n}\\
&\leq \sum_{n=1}^{N}e^{-cI -n\epsilon_{c}+n\log(d)}\\
&= e^{-cI}e^{ -\epsilon_{c}+\log(d)}\frac{e^{ -N\epsilon_{c}+N\log(d)}-1}{e^{ -\epsilon_{c}+\log(d)}-1}.
\end{align*}

It follows that
\[
\limsup_{N/c \to 0} \frac{1}{c}\log \left( \sum_{n=1}^{N}\sum_{x\in Fix_{n}}e^{cf^{n}-nP(cf)}\frac{k^{n}}{n}\right) \leq
-I - \frac{-N\epsilon_{c}+N\log(d)}{c} = -I.\]

\end{proof}

\end{document}